\documentclass[reqno,a4paper,twoside]{amsart}

\usepackage{tikz}
\usepackage{amsrefs,dsfont,amssymb,mathtools,enumitem}
\newtheorem{theorem}{Theorem}[section]
\newtheorem{lemma}[theorem]{Lemma}
\newtheorem{proposition}[theorem]{Proposition}
\newtheorem{corollary}[theorem]{Corollary}
\newtheorem{remark}[theorem]{Remark}
\newtheorem{definition}[theorem]{Definition}

\DeclareMathOperator{\Dir}{Dir}
\DeclareMathOperator{\Succ}{Succ}
\DeclareMathOperator{\Pred}{Pred}
\DeclareMathOperator{\Supp}{Supp}
\DeclareMathOperator{\en}{en}
\DeclareMathOperator{\st}{st}
\DeclareMathOperator{\HH}{H}
\DeclareMathOperator{\factors}{factors}

\begin{document}

\title[Homogeneous components in the Dixmier conjecture]{Number of homogeneous components of counterexamples to the Dixmier conjecture}

\author{Jorge A. Guccione}
\address{Departamento de Matem\'atica\\ Facultad de Ciencias Exactas y Naturales-UBA, Pabell\'on~1-Ciudad Universitaria\\ Intendente Guiraldes 2160 (C1428EGA) Buenos Aires, Argentina.}
\address{Instituto de Investigaciones Matem\'aticas ``Luis A. Santal\'o''\\ Pabell\'on~1-Ciudad Universitaria\\ Intendente Guiraldes 2160 (C1428EGA) Buenos Aires, Argentina.}
\email{vander@dm.uba.ar}

\author{Juan J. Guccione}
\address{Departamento de Matem\'atica\\ Facultad de Ciencias Exactas y Naturales-UBA\\ Pabell\'on~1-Ciudad Universitaria\\ Intendente Guiraldes 2160 (C1428EGA) Buenos Aires, Argentina.}
\address{Instituto Argentino de Matem\'atica-CONICET\\ Saavedra 15 3er piso\\ (\!C1083ACA\!) Buenos Aires, Argentina.}
\email{jjgucci@dm.uba.ar}

\thanks{Jorge A. Guccione and Juan J. Guccione were supported by CONICET PIP 2021-2023 GI,11220200100423CO and CONCYTEC-FONDECYT within the framework of the contest ``Proyectos de Investigaci\'on B\'asica 2020-01'' [contract number 120-2020-FONDECYT]}

\author{Christian Valqui}
\address{Pontificia Universidad Cat\'olica del Per\'u, Secci\'on Matem\'aticas, PUCP, Av. Universitaria 1801, San Miguel, Lima 32, Per\'u.}

\address{Instituto de Matem\'atica y Ciencias Afines (IMCA) Calle Los Bi\'ologos 245. Urb San C\'esar. La Molina, Lima 12, Per\'u.}
\email{cvalqui@pucp.edu.pe}

\thanks{Christian Valqui was supported by CONCYTEC-FONDECYT within the framework of the contest ``Proyectos de Investigaci\'on B\'asica 2020-01'' [contract number 120-2020-FONDECYT]}



\begin{abstract}
Assume that $P$ and $Q$ are elements of $A_1$ satisfying $[P,Q] = 1$. The Dixmier Conjecture for $A_1$ says that they always generate $A_1$. We show that if $P$ is a sum of not more than~$4$ homogeneous elements of $A_1$ then $P$ and $Q$ generate $A_1$, which generalizes the main result in~\cite{HT}.
\end{abstract}	
	
\maketitle

\tableofcontents

\section*{Introduction}

The first Weyl algebra $A_1$ over a characteristic zero field $K$ is generated by $X,Y$ with $[Y,X]=YX-XY=1$. In~\cite{D} Dixmier posed six questions. The first asks if every endomorphism of $A_1$ is an automorphism, i.e., if, for some $P,Q\in A_1$, we have $[P,Q]=1$, does it imply that $P$ and $Q$ generate $A_1$? The Dixmier conjecture generalizes this question and states that any endomorphism of the $n$-th Weyl algebra $A_n$ is an automorphism, for a characteristic zero field $K$. In the early~1980s, L. Vaserstein and V. Kac demonstrated that the generalized DC implies the Jacobian conjecture (refer to~\cite{BCW} for explicit verification of this result). In~2005, Yoshifumi Tsuchimoto established the stable equivalence between the Dixmier and Jacobian conjectures~\cite{T}, a result independently confirmed in~2007 by Alexei Belov-Kanel and Maxim Kontsevich~\cite{BK}, as well as by Pascal Kossivi Adjamagbo and Arno van den Essen~\cite{AE}. A concise proof of the equivalence between these two problems can be found in~\cite{B1}.

One approach to try to solve the conjecture for $A_1$ is the minimal counterexample strategy. We assume that the Dixmier conjecture is false and find properties that a minimal counterexample must satisfy. In~\cite{GGV1} we prove that for a minimal counterexample $(P,Q)$, the greatest common divisor of the total degrees of $P$ and~$Q$ is greater than~$15$. Another way to measure the size of elements in $A_1$ is the mass, as introduced in~\cite{BL}. Consider the $\mathds{Z}$-graduation of $A_1$ defining the $i$-th homogeneous component as $K[YX]X^i$ if $i\ge0$ and $K[YX]Y^{-i}$ if $i<0$. Then the mass $m(P)$, of an element $P\in A_1$, is the number of nonzero homogeneous components of $P$.  In~\cite{BL} it is shown that if both $P$ and $Q$ have mass lower than or equal to~$2$, then $(P,Q)$ cannot be a counterexample to the Dixmier conjecture. In~\cite{HT} this result is improved showing that if one of $P$, $Q$ has mass lower than or equal to~$2$, then it cannot be a counterexample. In the present paper we improve these results and show in Corollary~\ref{resultado principal}, that  if one of $P$, $Q$ has mass lower than or equal to~$4$, then $(P,Q)$ cannot be a counterexample to the Dixmier conjecture. For this we use the basic geometry description of the support of elements in $A_1$ introduced in~\cite{GGV1}, and some results of~\cite{HT}. We also use that the $(\rho,\sigma)$-leading term of $P$ can be described by an univariate polynomial, and the number of elements in the support of that polynomial yields a lower bound for $m(P)$ if $\rho+\sigma>0$ (See Remark~\ref{soporte da cota}). We also show in Proposition~\ref{potencia de tres} that under some conditions the support has at least~$5$ elements, which gives us the desired lower bound under that conditions. Then we analyse all possible cases for the leading term $\ell_{1,1}(P)$ and show that these conditions are satisfied in the relevant cases. Moreover, since the mass does not change by scalar extensions, we can assume without loss of generality, that $K$ is algebraically closed, and we do it (note that in many places this hypothesis is not used).

\section{Preliminaries}
We adopt the notations of~\cite{GGV2} for directions and leading terms. We define the set of directions by
\begin{equation*}
\mathfrak{V} \coloneqq  \{(\rho,\sigma)\in \mathds{Z}^2: \text{$\gcd(\rho,\sigma) = 1$} \}.
\end{equation*}
We also set $\mathfrak{V}_{>0}\coloneqq \{(\rho,\sigma)\in \mathfrak{V}: \rho+\sigma>0\}$. We assign to each direction its corresponding unit vector in $S^1$, and we define an {\em in\-ter\-val} in $\mathfrak{V}$ as the preimage under this map of an arc of $S^1$ that is not the whole circle. We consider each~in\-ter\-val endowed with the order that increases counterclockwise. Clearly $\mathfrak{V}_{>0}$ is an~in\-ter\-val, and the order is given by
\begin{equation}\label{orden de direcciones}
(\rho,\sigma)<(\rho',\sigma')\Longleftrightarrow (\rho,\sigma)\times (\rho',\sigma')>0,
\end{equation}
where $(a,b)\times (c,d)\coloneqq ad-bc$. This order extends to the closed interval $\mathfrak{V}_{\ge 0}$, setting $(1,-1)<(\rho,\sigma)<(-1,1)$ for $(\rho,\sigma)\in\mathfrak{V}_{>0}$. For $(i,j)\in \mathds{Z}^2$ and $(\rho,\sigma)\in \mathfrak{V}$, we set $v_{\rho,\sigma}(i,j)\coloneqq \rho i+\sigma j$ and for $R=\sum a_{ij}x^i y^j\in L\coloneqq K[x,y]$, we consider the valuation $v_{\rho,\sigma}(R)\coloneqq \max\left\{v_{\rho,\sigma}(i,j): a_{ij}\ne 0\right\}$. Note that $v_{\rho,\sigma}(R) = -\infty$ if and only if $R=0$. For $R\in L$, we set
$$
\ell_{\rho,\sigma}(R)\coloneqq \displaystyle{\sum_{\{\rho i +\sigma j = v_{\rho,\sigma}(R)\}}} a_{ij} x^{i} y^j.
$$
We consider the linear isomorphism $\psi\colon A_1\to L$, which sends $X^iY^j$ to $x^i y^j$, and for $P\in A_1$, we define
$$
v_{\rho,\sigma}(P)\coloneqq v_{\rho,\sigma}(\psi(P))\quad\text{and}\quad \ell_{\rho,\sigma}(P)\coloneqq \ell_{\rho,\sigma}(\psi(P))\in L.
$$
Let $P\in A_1\setminus\{0\}$. Then the support of $P=\sum_{i,j} a_{ij} X^iY^j$ is
$$
\Supp(P)\coloneqq \{(i,j)\in \mathds{N}_0: a_{ij}\ne 0\}.
$$
For each $P\in A_1\setminus\{0\}$, we let $\HH(P)$ denote the convex hull of the support of $P$. As it is well known, $\HH(P)$ is a polygon, called the {\em Newton polygon of $P$}, and it is evident that each one of its edges is the convex hull of the support of $\ell_{\rho,\sigma}(P)$, where $(\rho,\sigma)$ is orthogonal to the given edge and points outside of $\HH(P)$.

\smallskip

Let $(\rho,\sigma)\in \mathfrak{V}$ arbitrary and $P\ne 0$. We let $\st_{\rho,\sigma}(P)$ and $\en_{\rho,\sigma}(P)$ denote the first and the last point that we find on $H(\ell_{\rho,\sigma}(P))$ when we run counterclockwise along the boundary of $\HH(P)$. Note that these points coincide when $\ell_{\rho,\sigma}(P)$ is a monomial.

\begin{remark}\label{soporte da cota} Note that the $\mathds{Z}$-graduation on $A_1$ is induced by the $v_{1,-1}$ valuation. In fact, for $i\ge 0$, an element $P\in A_1\setminus\{0\}$ lies in $K[YX]X^i$ if and only if $P$ is $(1,-1)$-homogeneous and $v_{1,-1}(P)=i$, and similarly, an element $P\in A_1\setminus\{0\}$ lies in $K[YX]Y^i$ if and only if $P$ is $(1,-1)$-homogeneous and $v_{1,-1}(P)=-i$. Assume that $(\rho,\sigma)\in \mathfrak{V}_{>0}$ and let $P\in A_1\setminus\{0\}$. Then, we have
$$
\st_{\rho,\sigma}(P)=\Supp(\ell_{1,-1}(\ell_{\rho,\sigma}(P))) \quad\text{and}\quad \en_{\rho,\sigma}(P)=\Supp(\ell_{-1,1}(\ell_{\rho,\sigma}(P))).
$$
If we now write
$$
\ell_{\rho,\sigma}(P)=x^i y^j \sum_{l=0}^n a_l x^{-\sigma l}y^{\rho l},
$$
with $a_0\ne 0$ and $a_n\ne 0$, then $v_{1,-1}(a_l x^{-\sigma l}y^{\rho l})=-l(\rho+\sigma)$, and so
\begin{equation}\label{calculo de st y en}
\st_{\rho,\sigma}(P) = (i,j) \qquad\text{and}\qquad \en_{\rho,\sigma}(P) = (i,j)+n(-\sigma,\rho).
\end{equation}
When additionally $\rho>0$, we define $f_P(y) = f_{P,\rho,\sigma}(y)\coloneqq \sum_{l=0}^n a_l y^{\rho l}\in K[y]$, which is the poly\-nomial $f_{P,\rho,\sigma}^{(1)}$ introduced in~\cite{GGV1}*{Definition~2.8}.
Note that
\begin{equation}\label{relacion entre ell y f}
\deg(f_P) = n\rho\qquad\text{and}\qquad \ell_{\rho,\sigma}(P) = x^i y^j f_P(z)\quad\text{where $z\coloneqq x^{-\sigma/\rho}y$.}
\end{equation}
For each univariate polynomial $f$, we set $t(f)\coloneqq \#\Supp(f)$. Since $\rho+\sigma>0$, each monomial $a_l x^{-\sigma l}y^{\rho l}$, with $a_l\ne 0$, is in a different $(1,-1)$-homogeneous component of $\ell_{\rho,\sigma}(P)$, and so
\begin{equation}\label{m contra t}
m(P)\ge m(\ell_{\rho,\sigma}(P)) = \#\Supp(\ell_{\rho,\sigma}(P)) = t(f_P).
\end{equation}
\end{remark}

For $P\in A_1\setminus\{0\}$, we define
$$
\Dir(P)\coloneqq \{(\rho,\sigma)\in\mathfrak{V}:\#\Supp(\ell_{\rho,\sigma}(P))>1\}.
$$
Suppose that $P\in A_1$ is not a monomial and let $(\rho,\sigma)\in\mathfrak{V}$ arbitrary. We define the {\em successor} $\Succ_P(\rho,\sigma)$ of $(\rho,\sigma)$ to be the first element of $\Dir(P)$ that one encounters starting from $(\rho,\sigma)$ and running counterclockwise, and the {\em predecessor} $\Pred_P(\rho,\sigma)$, to be the first one, if we run clockwise.

\begin{proposition} \label{no estan en D mayor o igual a cero} Assume that $(P,Q)$ is a counterexample to the DC (this means that $P$ and $Q$ do not generate $A_1$ and that $[P,Q] = 1$). Then, we have $v_{1,-1}(P)>0$ and $v_{-1,1}(P)>0$.
\end{proposition}

\begin{proof} By~\cite{HT}*{Theorem~3.7} we know that $P$ cannot be in
$$
D_{\le 0}\coloneqq \{R\in A_1, v_{1,-1}(R)\le 0\}\quad\text{nor in}\quad D_{\ge 0}\coloneqq\{R\in A_1, v_{-1,1}(R)\le 0\}.
$$
Thus $v_{1,-1}(P)>0$ and $v_{-1,1}(P)>0$.
\end{proof}

\begin{remark} By the previous proposition, if $(P,Q)$ is a counterexample to the DC, then $P\notin K[X]\cup K[Y]$ and $P$ is not a monomial. We will use freely these facts.
\end{remark}

\begin{proposition}\label{v mayor que cero} If $(P,Q)$ is a counterexample and $(\rho,\sigma)\!\in\! \mathfrak{V}_{>0}$, then $v_{\rho,\sigma}(P)\!>\!0$.
\end{proposition}

\begin{proof} By Proposition~\ref{no estan en D mayor o igual a cero} we know that $v_{1,-1}(P)>0$ and $v_{-1,1}(P)>0$. Consequently there are points $(i,j),(k,l)\in \Supp(P)$ with $i>j$ and $k<l$.

\begin{itemize}

\item[-] If $\rho>0$ and $\sigma\ge 0$, then $v_{\rho,\sigma}(P)\ge v_{\rho,\sigma}(i,j)=\rho i+\sigma j \ge \rho i >0$,

\item[-] if $\rho>0$ and $\sigma< 0$, then $v_{\rho,\sigma}(P)\ge v_{\rho,\sigma}(i,j)=\rho i+\sigma j \ge \rho i+\sigma i>0$,

\item[-] if $\sigma>0$ and $\rho\ge 0$, then $v_{\rho,\sigma}(P)\ge v_{\rho,\sigma}(k,l)=\rho k+\sigma l \ge \sigma l>0$,

\item[-] if $\sigma>0$ and $\rho< 0$, then $v_{\rho,\sigma}(P)\ge v_{\rho,\sigma}(k,l)=\rho k+\sigma l \ge \rho l+\sigma l>0$,

\end{itemize}
as desired.
\end{proof}

\section{Support of univariate polynomials}

Let $f\in K[x]\setminus\{0\}$. Then the support of $f=\sum_ja_j x^j$ is
$$
\Supp(f)\coloneqq \{j\in \mathds{N}_0: a_j\ne 0\}.
$$
The number of non-zero terms of $f$ is called $t(f)\coloneqq \# \Supp(f)$. We consider the equivalence relation on $K[x]$, generated by
\begin{itemize}

\item[-] $f\simeq \lambda f$ for all $\lambda\in K^{\times}$,

\item[-] $f(x)\simeq f(\lambda x)$ for all $\lambda\in K^{\times}$,

\item[-] $f(x)\simeq f(x^k)$ for all $k\in \mathds{N}$,

\item[-] $f(x)\simeq x^n f(1/x)$, where $n = \deg(f)$.
\end{itemize}
Clearly, if $f\simeq g$, then $f^k\simeq g^k$ and $t(f^k)=t(g^k)$, for all $k\in\mathds{N}$. Furthermore, since $K$ is algebraically closed, every polynomial $f$ with $f(0)\ne 0$, is equivalent to a polynomial of the form $g=1+x^j+\sum_{i=j+1}^n a_i x^i$. We will use these facts freely.

\begin{remark}\label{facilongo} Note that if $f = \sum_{i=0}^n a_i x^i$ with $a_n\ne 0$, then $x^n f(1/x) = \sum_{i=0}^n a_{n-i} x^{i}$.
\end{remark}

\begin{remark}\label{raiz k esima} Let $p = 1 + xq$ with $q\in K[[x]]$ and let $k\in \mathds{N}$. It is well known that
$$
u\coloneqq (1+xq)^{1/k} = \sum_{i\ge 0} \binom{1/k}{i} (xq)^i
$$
is the unique formal power series such that $u(0) = 1$ and $u^k = p$.
\end{remark}

\begin{proposition}\label{potencia de tres} Let $f\in K[x]$ such that $t(f)\ge 3$. Then, for all $k\ge 2$, we have
\begin{enumerate}[itemsep=0.7ex, topsep=1.0ex, label={\emph{\arabic*)}}]

\item $t(f^k)\ge 4$.

\item If $t(f^k)= 4$, then $k=2$ and $f\simeq 1+x-\frac 12 x^2$.
\end{enumerate}
\end{proposition}

\begin{proof} 1)\enspace Let $f=\sum a_i x^i$. Since $t(f)\ge 3$, replacing $f$ by an equivalent polynomial, we can assume that
$$
f= 1+\sum_{i=j}^n a_i x^i\quad\text{with $n>j$, $a_j=1$ and $a_n\ne 0$.}
$$
Let $l\coloneqq  \max\{i<n : a_i\ne 0\}$. Then
\begin{equation*}
f= 1+x^j+\sum_{i=j+1}^l a_i x^i+a_n x^n,
\end{equation*}
where $0<j\le l <n$. For all $k\ge 2$, we have
$$
f^k=1+k x^j+\sum_{i=j+1}^{n(k-1)+l-1} b_i x^i+ka_n^{k-1}a_lx^{n(k-1)+l}+a_n^k x^{nk},
$$
where $b_i$ is the coefficient of $x^i$ in $f^k$. Since $k\ge 2$, we have $0<j< n(k-1)+l<kn$, and so $\{0,j,n(k-1)+l,kn\}\subseteq \Supp(f^k)$, which implies that $t(f^k)\ge 4$.

\smallskip

\noindent 2)\enspace If $t(f^k)=4$, then $\Supp(f^k)=\{0,j,n(k-1)+l,kn\}$. Since
$$
f^k=1+k x^j+ka_n^{k-1}a_lx^{n(k-1)+l}+a_n^k x^{nk},
$$
we have $a_l\ne 0$ and $a_n\ne 0$. Moreover, by Remark~\ref{raiz k esima},
\begin{equation*}\label{nueva expresion f}
f = \sum_{i=0}^n a_i x^i =\sum_{i\ge 0} \binom{1/k}{i} \bigl(k x^j+ka_n^{k-1}a_lx^{n(k-1)+l}+a_n^k x^{nk}\bigr)^i
\end{equation*}
Comparing this with
\begin{equation*}\label{expresion h}
h\coloneqq (1+kx^j)^{1/k} = \sum_{s\ge 0} \binom{1/k}{s} k^sx^{js}=\sum_{i\ge 1}h_i x^i\in K[[x]]
\end{equation*}
and using that $n(k-1)+l>n$, we obtain that
$$
a_i = h_i\text{ for $0\le i< n(k-1)+l$}\qquad\text{and}\qquad  h_{n(k-1)+l}\ne a_{n(k-1)+l} = 0.
$$
Consequently,
$$
1+x^j+\sum_{i=j+1}^l a_i x^i+a_n x^n = f = \sum_{i=0}^{n(k-1)+l-1} h_ix^i\qquad\text{and}\qquad h_{n(k-1)+l}\ne 0.
$$
Since $h_i\ne 0$ if and only if $j\mid i$ and $n(k-1)+l>n$, this implies that
$$
l = qj,\quad n = (q+1)j\quad\text{and}\quad n(k-1)+l = (q+2)j\qquad\text{for some $q\in \mathds{N}$.}
$$
Thus $(q+1)j(k-1) + qj = (q+2)j$, which implies that $q=1$ and $k=2$.~Con\-se\-quently $l=j$ and $n=2j$, and so $\Supp(f)=\{0,j,n\}$. Since
$$
a_n=h_n=h_{2j}= \binom{1/2}{2}2^2=-\frac 12,
$$
we obtain $f=1+x^j-\frac 12 x^{2j}\simeq 1+x-\frac 12 x^2$, as desired.
\end{proof}

\begin{proposition}\label{factores lineales distintos} Let us assume that $f(0)\ne 0$ and $\deg(f)=n$. If $f\simeq 1+x-\frac{1}{2} x^2$ or $t(f)=2$, then $f$ has $n$ different linear factors in $K$ (recall that $K$ is algebraically closed).
\end{proposition}

\begin{proof} It suffices to show that $\gcd(f,f') = 1$. Note that if $t(f) = 2$, then $f\simeq 1+x$. Since $\gcd(g,g') = 1$, for $g=1+x$ and $g=1+x-\frac{1}{2} x^2$, we are reduced to prove that if $f\simeq g$ and $f(0)\ne 0$, then $\gcd(f,f') = 1$ if and only if $\gcd(g,g') = 1$. For this we must consider each one of the four cases in the equivalence relation defined above Remark~\ref{facilongo}. We leave the details to the reader.
\end{proof}

Remember that $\psi\colon A_1\to L$ is the linear isomorphism which sends $X^iY^j$ to $x^i y^j$. Remember also from~\cite{GGV1}*{Definition~2.2}, that, for $P,Q\in A_1\setminus\{0\}$ and $(\rho,\sigma)\in \mathfrak{V}_{>0}$,
$$
[P,Q]_{\rho,\sigma}\coloneqq \begin{cases} 0 &\text{if $v_{\rho,\sigma}([P,Q]) < v_{\rho,\sigma}(P) + v_{\rho,\sigma}(Q) - (\rho+\sigma)$,}\\ \ell_{\rho,\sigma}([P,Q]) &\text{if $v_{\rho,\sigma}([P,Q]) = v_{\rho,\sigma}(P) + v_{\rho,\sigma}(Q) - (\rho+\sigma)$.}\end{cases}
$$
Note that $[P,Q]_{\rho,\sigma} = 0$ if and only if $v_{\rho,\sigma}([P,Q]) < v_{\rho,\sigma}(P) + v_{\rho,\sigma}(Q) - (\rho+\sigma)$.

\begin{lemma}\label{pepe0} For each $R,Q\in A_1\setminus\{0\}$ and $k\in \mathds{N}$ the followings facts hold:

\begin{enumerate}[itemsep=0.7ex, topsep=1.0ex, label={\emph{\arabic*)}}]

\item $\ell_{\rho,\sigma}([R^k,Q]) = k\ell_{\rho,\sigma}(R)^{k-1}\ell_{\rho,\sigma}([R,Q])$,

\item If $[R^k,Q]_{\rho,\sigma} \ne 0$, then $[R,Q]_{\rho,\sigma} \ne 0$.

\end{enumerate}
\end{lemma}

\begin{proof} To begin with note that
$$
\ell_{\rho,\sigma}([R^k,Q]) = \ell_{\rho,\sigma}\left(\sum R^i[R,Q]R^{k-i-1}\right) = k\ell_{\rho,\sigma}(R)^{k-1}\ell_{\rho,\sigma}([R,Q]),
$$
where the last equality follows easily from~\cite{GGV1}*{Proposition~1.9(2)}. So item~1) holds. Hence, by~\cite{GGV1}*{Proposition~1.9(3)}, we have
\begin{equation}\label{pepe3}
v_{\rho,\sigma}([R^k,Q]) = (k-1)v_{\rho,\sigma}(R) + v_{\rho,\sigma}([R,Q]).
\end{equation}
Moreover $v_{\rho,\sigma}([R^k,Q]) = v_{\rho,\sigma}(R^k) + v_{\rho,\sigma}(Q) - (\rho+\sigma)$, because $[R^k,Q]_{\rho,\sigma} \ne 0$. Since $v_{\rho,\sigma}(R^k) = kv_{\rho,\sigma}(R)$, from this and~\eqref{pepe3}, we obtain that
$$
v_{\rho,\sigma}([R,Q]) = v_{\rho,\sigma}(R) + v_{\rho,\sigma}(Q) - (\rho+\sigma).
$$
Hence $[R,Q]_{\rho,\sigma} \ne 0$, as desired.
\end{proof}

\begin{proposition}\label{F existe} Let $(P,Q)$ be a counterexample to the DC and let $(\rho,\sigma)\in \mathfrak{V}_{>0}$. Then $\ell_{\rho,\sigma}(P)=\mu \psi(R)^k$ for some $\mu\in K^{\times}$, some $k\ge 2$ and some $(\rho,\sigma)$-ho\-mo\-ge\-neous element $R\in A_1\setminus\{0\}$. Moreover, there exists a $(\rho,\sigma)$-ho\-mogeneous ele\-ment \hbox{$F\in A_1$}, such that
\begin{equation}\label{RF}
v_{\rho,\sigma}(F)=\rho+\sigma\quad\text{and}\quad [R,F]_{\rho,\sigma}=\psi(R).
\end{equation}
\end{proposition}

\begin{proof} By Proposition~\ref{v mayor que cero}, we have $v_{\rho,\sigma}(P)>0$, and so, by~\cite{GGV1}*{Theorem~4.1}, there exists $(\rho,\sigma)$-homogeneous element $\tilde F\in A_1\setminus\{0\}$, such that
\begin{equation}\label{pepe1}
v_{\rho,\sigma}(\tilde F)=\rho+\sigma\quad\text{and}\quad[P,\tilde F]_{\rho,\sigma}=\ell_{\rho,\sigma}(P).
\end{equation}
Moreover, by~\cite{HT}*{Corollary~2.6 and Theorem~3.13} there exists a $(\rho,\sigma)$-homogeneous element $R\in A_1\setminus\{0\}$ such that
\begin{equation}\label{pepe2}
\ell_{\rho,\sigma}(P)=\mu \psi(R)^k\qquad\text{for some $\mu\in K^{\times}$ and $k\ge 2$.}
\end{equation}
Hence $\ell_{\rho,\sigma}(P)= \mu \psi(R)^k = \mu \ell_{\rho,\sigma}(R)^k = \ell_{\rho,\sigma}(\mu R^k)$ (by~\cite{GGV1}*{Proposition~1.9(2)}), and consequently, by~\cite{GGV1}*{Corollary~2.6} and equalities~\eqref{pepe1} and~\eqref{pepe2},
$$
\mu[R^k,\tilde F]_{\rho,\sigma} = [\mu R^k,\tilde F]_{\rho,\sigma} = [P,\tilde F]_{\rho,\sigma} = \ell_{\rho,\sigma}(P) = \mu\psi(R)^k\ne 0.
$$
This, combined with Lemma~\ref{pepe0}(1), gives
$$
[R^k,\tilde F]_{\rho,\sigma} = \ell_{\rho,\sigma}([R^k,\tilde F]) = k\ell_{\rho,\sigma}(R)^{k-1}\ell_{\rho,\sigma}([R,\tilde F]).
$$
Moreover $[R^k,\tilde F]_{\rho,\sigma}\ne 0$ implies $[R,\tilde F]_{\rho,\sigma}\ne 0$ (by Lemma~\ref{pepe0}(2)). Hence
$$
\mu k \psi(R)^{k-1}[R,\tilde F]_{\rho,\sigma} = \mu k\ell_{\rho,\sigma}(R)^{k-1}\ell_{\rho,\sigma}([R,\tilde F]) = \mu [R^k,\tilde F]_{\rho,\sigma}=\mu \psi(R)^k,
$$
and so $k[ R,\tilde F]_{\rho,\sigma}=\psi(R)$, since $\psi( R)^{k-1}$ is not a zero divisor. Setting $F\coloneqq k \tilde F$ we obtain $[R,F]_{\rho,\sigma}= k [R,\tilde F]_{\rho,\sigma}=\psi(R)$, as desired.
\end{proof}

\section{Cases}

\begin{definition} For $A,B\in \mathds{Q}^2$ we write $A\sim B$, if $A=\lambda B$ for some $\lambda\in \mathds{Q}^{\times}$.
\end{definition}

\begin{remark} For $A,B\in \mathds{Q}^2$, we have
$$
A\times B=0 \Longleftrightarrow \left(A=0,\ B=0 \text{ or } A\sim B\right).
$$
In this case, we say that $A$ and $B$ are {\em aligned}. We write $A\nsim B$, if $A\times B\ne 0$. So
$$
A\nsim B \Longleftrightarrow \left(A\ne 0,\ B\ne 0\text{ and }A\ne \lambda B \text{ for all $\lambda\in \mathds{Q}$}\right).
$$
\end{remark}

\begin{remark}\label{casos soporte} Let $(P,Q)$ be a counterexample to the DC. By Proposition~\ref{v mayor que cero}, we have $v_{1,1}(P)>0$, and so, by~\cite{GGV1}*{Theorem~4.1}, there exists~a~$(1,1)$-ho\-mogeneous element $F\in A_1$ such that $[P,F]_{1,1}=\ell_{1,1}(P)$ and $v_{1,1}(F)=1+1=2$. But then
$$
\Supp(F)\subseteq \bigl\{(i,j)\in \mathds{N}_0^2: 2=v_{1,1}(i,j)=i+j\}=\{(2,0),(1,1),(0,2)\bigr\}.
$$
Set $z\coloneqq x^{-1}y$ and write $\ell_{1,1}(P)=x^i y^j f_{P}(z)$ and $\psi(F)=x^uy^v f_F(z)$. Since $\Supp(F)\subseteq \{(0,2),(1,1),(2,0)\}$, by~\cite{GGV1}*{Corollary~4.4(1)}, we have \hbox{$\# \factors(f_P)\le 2$}, where $\# \factors(f_P)$ denotes the number of different linear factors of $f_P$.

\begin{enumerate}[itemsep=0.7ex, topsep=1.0ex, label={\emph{\arabic*)}}]

\item If $\# \factors(f_P) = 0$, then $\ell_{1,1}(P)$ is a monomial with support $\{(i,j)\}$, and we have three possibilities,
\begin{itemize}
\item[a)] $i>0$, $j>0$,

\item[b)] $i=0$, $j>0$,

\item[c)] $i>0$, $j=0$.
\end{itemize}

\noindent
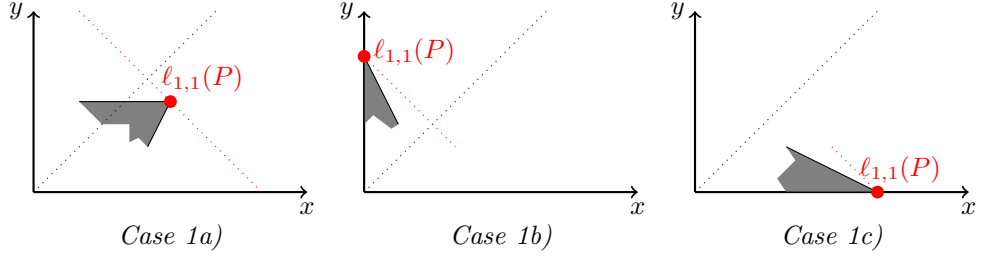
\begin{tikzpicture}[scale=0.6]
\draw [thick,->] (0,0)--(6,0) node[anchor=north]{$x$};
\draw [thick,->] (0,0) --(0,4) node[anchor=east]{$y$};
\fill[gray] (1,2) -- (3,2) -- (2.5,1)--(2.3,1.2)--(2.1,1.1)--(2.1,1.5)--(1.5,1.5);
\draw[-] (1,2) -- (3,2) -- (2.5,1);
\draw[dotted] (0,0) -- (4,4);
\draw[dotted,red] (1,4) -- (5,0);
\fill[red] (3,2) circle (4pt);
\draw(3,-1)node{Case 1a)};
\draw[red](3.7,2.5)node{$\ell_{1,1}(P)$};
\end{tikzpicture}
\begin{tikzpicture}[scale=0.6]
\draw [thick,->] (0,0)--(6,0) node[anchor=north]{$x$};
\draw [thick,->] (0,0) --(0,4) node[anchor=east]{$y$};
\fill[gray] (0,1.5) -- (0,3) -- (0.75,1.5)--(0.6,1.4)--(0.2,1.7);
\draw[-] (0,1.5) -- (0,3) -- (0.75,1.5);
\draw[dotted] (0,0) -- (4,4);
\draw[dotted,red] (0,3) -- (2,1);
\fill[red] (0,3) circle (4pt);
\draw(3,-1)node{Case 1b)};
\draw[red](1.1,3.1)node{$\ell_{1,1}(P)$};
\end{tikzpicture}
\begin{tikzpicture}[scale=0.6]
\draw [thick,->] (0,0)--(6,0) node[anchor=north]{$x$};
\draw [thick,->] (0,0) --(0,4) node[anchor=east]{$y$};
\fill[gray] (2,1) -- (4,0) -- (2,0)--(1.8,0.3)--(2.2,0.7);
\draw[-] (2,1) -- (4,0) -- (2,0);
\draw[dotted] (0,0) -- (4,4);
\draw[dotted,red] (3,1) -- (4,0);
\fill[red] (4,0) circle (4pt);
\draw(3,-1)node{Case 1c)};
\draw[red](4.5,0.5)node{$\ell_{1,1}(P)$};
\end{tikzpicture}

\item If $\# \factors(f_P)=1$, then, by~\cite{GGV1}*{Corollary~4.4(1)}, $f_F$ has at least one linear factor, and so $\# \Supp(F)\ge 2$. We have the following possibilities:
\begin{itemize}

\item[a)] $\st_{1,1}(P)\sim (2,0)$ and $\en_{1,1}(P)\sim (0,2)$,

\item[b)] $\st_{1,1}(P)\sim (2,0)$ and $\en_{1,1}(P)\nsim (0,2)$,

\item[c)] $\st_{1,1}(P)\nsim (2,0)$ and $\en_{1,1}(P)\sim (0,2)$.
\end{itemize}
In fact, if~a) and~b) are not satisfied, then $\st_{1,1}(P)\nsim (2,0)$. Consequently, by~\cite{GGV1}*{Theorem~4.1(1)} necessarily $\st_{1,1}(F) = (1,1)$. Since $\# \Supp(F)\ge 2$, this implies that $\en_{1,1}(F) = (0,2)$, and so, $\en_{1,1}(P)\sim (0,2)$, by~\cite{GGV1}*{Theo\-rem~4.1(2)}.

\noindent 
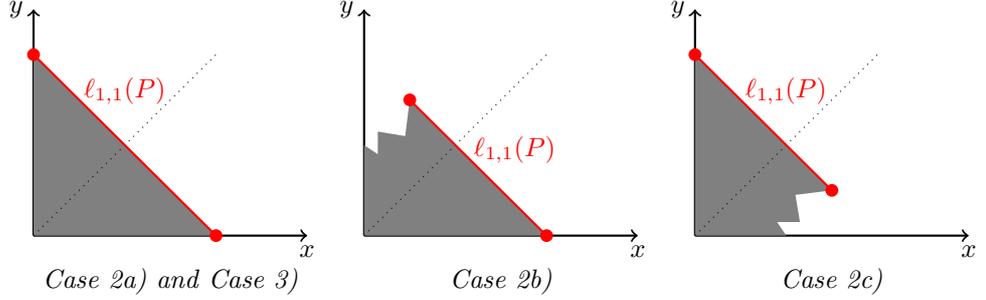
\begin{tikzpicture}[scale=0.6]
\draw [thick,->] (0,0)--(6,0) node[anchor=north]{$x$};
\draw [thick,->] (0,0) --(0,5) node[anchor=east]{$y$};
\fill[gray] (0,0) -- (0,4) -- (4,0);
\draw[dotted] (0,0) -- (4,4);
\draw[red,thick] (0,4) -- (4,0);
\fill[red] (0,4) circle (4pt)
(4,0) circle (4pt);
\draw(3,-1)node{Case 2a) and Case 3)};
\draw[red](2,3.2)node{$\ell_{1,1}(P)$};
\end{tikzpicture}
\begin{tikzpicture}[scale=0.6]
\draw [thick,->] (0,0)--(6,0) node[anchor=north]{$x$};
\draw [thick,->] (0,0) --(0,5) node[anchor=east]{$y$};
\fill[gray] (0,0) -- (0,2)--(0.3,1.8)--(0.3,2.3)--(0.9,2.2)--(1,3) -- (4,0);
\draw[dotted] (0,0) -- (4,4);
\draw[red,thick] (1,3) -- (4,0);
\fill[red] (1,3) circle (4pt)
(4,0) circle (4pt);
\draw(3,-1)node{Case 2b)};
\draw[red](3.3,1.9)node{$\ell_{1,1}(P)$};
\end{tikzpicture}
\begin{tikzpicture}[scale=0.6]
\draw [thick,->] (0,0)--(6,0) node[anchor=north]{$x$};
\draw [thick,->] (0,0) --(0,5) node[anchor=east]{$y$};
\fill[gray] (0,0) -- (2,0)--(1.8,0.3)--(2.3,0.3)--(2.2,0.9)--(3,1) -- (0,4);
\draw[dotted] (0,0) -- (4,4);
\draw[red,thick] (0,4) -- (3,1);
\fill[red] (0,4) circle (4pt)
(3,1) circle (4pt);
\draw(3,-1)node{Case 2c)};
\draw[red](2,3.2)node{$\ell_{1,1}(P)$};
\end{tikzpicture}

\item If $\# \factors(f_P)=2$, then $f_F$ has also two different linear factors, and consequently $(0,2),(2,0)\in \Supp(F)$. Again by items~(1) and~(2) of~\cite{GGV1}*{The\-orem~4.1}, necessarily $\st_{1,1}(P)\sim (2,0)$ and $\en_{1,1}(P)\sim (0,2)$.
\end{enumerate}
\end{remark}

\section{Lower bound for $m(P)$}

Recall from~\cite{GGV1} that $P\in A_1\setminus \{0\}$ is {\em subrectangular with vertice $(a,b)\in \mathds{N}\times \mathds{N}$} if
$$
(a,b)\in \Supp(P) \subseteq \{(i,j):\text{ $0\le i\le a$ and $0\le j\le b$}\}.
$$

We will see in Theorem~\ref{central} that the following proposition covers all the cases~of~Re\-mark~\ref{casos soporte}.

\begin{proposition}\label{proposicion principal} Let $(P,Q)$ be a counterexample to the DC. We have:

\begin{enumerate}[itemsep=0.7ex, topsep=1.0ex, label={\emph{\arabic*)}}]

\item If $P$ is subrectangular and $v_{1,-1}(\en_{1,0}(P))<0$, then $m(P)\ge 5$.

\item If $\ell_{1,1}(P)=\lambda y^n$ and $v_{1,-1}(\en_{1,0}(P))<0$, then $m(P)\ge 5$.

\item If $\ell_{1,1}(P)=\lambda y^n$ and $v_{1,-1}(\en_{1,0}(P))>0$, then $m(P)\ge 10$.

\item If $\# \factors(f_P)\!=\!1$, $\st_{1,1}(P)\!\sim\!(2,0)$ and $\en_{1,1}(P)\!\sim\! (0,2)$, then \hbox{$m(P)\!>\!16$}.

\item If $\# \factors(f_P)\!=\!1$, $v_{1,-1}(\en_{1,0}(P))\!<\!0$, $\en_{1,0}(P)\!=\!\st_{1,1}(P)\!\nsim\! (2,0)$ and $\en_{1,1}(P)\sim (0,2)$, then $m(P)\ge 5$.

\item If $\# \factors(f_P)=1$, $v_{1,-1}(\en_{1,0}(P))>0$, $\en_{1,0}(P)=\st_{1,1}(P)\nsim (2,0)$ and $\en_{1,1}(P)\sim (0,2)$, then $m(P)\ge 10$.

\item If $\# \factors(f_P)\!=\!2$, $\st_{1,1}(P)\!\sim\! (2,0)$ and $\en_{1,1}(P)\!\sim\! (0,2)$, then \hbox{$m(P)\!\ge\! 5$}.
\end{enumerate}
\end{proposition}

\begin{proof} Along the proof we will use several times~\cite{GGV1}*{Theorem~4.1}, which applies by Proposition~\ref{v mayor que cero}. We will prove first the easier cases. Note that, by~\cite{GGV1}*{Corollary~7.4}, we know that $\deg(P)=v_{1,1}(P)\ge 16$; while, by~\cite{GGV1}*{Theorem~4.1}, we have $v_{1,-1}(\en_{1,0}(P))\ne 0$. Moreover if $\# \factors(p)=1$, then there exists $i,j\in \mathds{N}_0$, $k\in \mathds{N}$ and $\lambda,\mu\in K^{\times}$ such that $\ell_{1,1}(P) = x^i y^j \lambda (z-\mu)^k$, where $z\coloneqq x^{-1}y$.

\smallskip

\noindent\textsc{Case 4)}: By~\eqref{calculo de st y en}, in this case
$$
(i,j) = \st_{1,1}(P)\sim (2,0) \quad\text{and}\quad  (i,j)+k(-1,1) = \en_{1,1}(P) \sim (0,2).
$$
Hence $j=0$ and $i=k$. Consequently $\ell_{1,1}(P)= \lambda(y-\mu x)^k$, which implies that $k=\deg(P)\ge 16$. So, by Remark~\ref{soporte da cota}, we have $m(P)\ge t(\lambda(z-\mu)^k) = k+1>16$.

\smallskip

\noindent\textsc{Case 6)}: By~\eqref{calculo de st y en}, in this case
$$
(i,j) = \st_{1,1}(P) = \en_{1,0}(P) \quad\text{and}\quad  (i,j)+k(-1,1) = \en_{1,1}(P) \sim (0,2).
$$
Hence $i-j = v_{1,-1}(\en_{1,0}(P))>0$ and $i=k$. So $j<k$ and $\ell_{1,1}(P)= \lambda y^j (y-\mu x)^k$. Thus $k>\deg(P)/2\ge 8$, and consequently, $m(P)\ge t(\lambda(z-\mu)^k) = k+1>9$.

\smallskip

\noindent\textsc{Case 3)}: Let $\tau\colon A_1\to A_1$ be the morphism given by $\tau(X)\coloneqq Y$ and $\tau(Y)\coloneqq -X$. Set $P_0\coloneqq \tau(P)$. Clearly $(P_0,\tau(Q))$ is a counterexample to DC, $m(P_0)=m(P)$, $\ell_{1,1}(P_0)=(-1)^n\lambda x^n$ and $v_{1,-1}(\st_{0,1}(P_0))<0$. Define now $(\rho_0,\sigma_0)\coloneqq \Succ_{P_0}(1,1)$. Then, $(1,1)<(\rho_0,\sigma_0)<(0,1)$, since otherwise $P_0\in K[X]$. Hence, $\sigma_0>\rho_0>0$, and so, by~\cite{GGV1}*{Lemma~6.4}, we have $\rho_0=1$, $\st_{1,\sigma_0}(P_0) = (n_0,0)$ where $n_0\coloneqq v_{1,1}(P_0)$, and
$$
\ell_{1,\sigma_0}(P_0)\! =\! x^{n_0}f_{P_0,1,\sigma_0}(x^{-\sigma_0}y)\! =\! x^{n_0} \lambda_0(x^{-\sigma_0}y-\mu_0)^{k_0}\! =\! x^{n_0-\sigma_0 k_0} \lambda_0(y-\mu_0 x^{\sigma_0})^{k_0},
$$
where $\lambda_0,\mu_0\in K^{\times}$ and $k_0\in \mathds{N}$. Note that $m(P) = m(P_0)\ge t(f_{P_0,1,\sigma_0})=k_0+1$ (by Remark~\ref{soporte da cota}). So, in order to finish the proof it suffices to show that $k_0\ge 9$. For this we will prove that
\begin{equation}\label{pepe4}
k_0>n_0-k_0\sigma_0 \qquad\text{and}\qquad n_0-k_0\sigma_0+k_0\ge 16.
\end{equation}
By~\eqref{calculo de st y en}, we have
\begin{equation*}
\st_{1,\sigma_0}(P_0) = (n_0,0) \qquad\text{and}\qquad \en_{1,\sigma_0}(P_0) = (n_0,0)+k_0(-\sigma_0,1).
\end{equation*}
We also have $\en_{1,\sigma_0}(P_0)=\st_{0,1}(P_0)$. In fact, otherwise there exists $(\tilde\rho,\tilde \sigma)\in \Dir(P_0)$ such that
$$
(1,1)<(1,\sigma_0)<(\tilde\rho,\tilde\sigma)<(0,1)\quad\text{and}\quad \st_{\tilde\rho,\tilde \sigma}(P_0)=\en_{1,\sigma_0}(P_0)=(n_0-k_0\sigma_0,k_0).
$$
Hence $\tilde \sigma > \tilde \rho>0$, and consequently by~\cite{GGV1}*{Lemma~6.4(1)}, we have $k_0=0$. But we know that $k_0>0$, and so, $\en_{1,\sigma_0}(P_0)=\st_{0,1}(P_0)$, as we want. Since
$$
v_{1,-1}(n_0-k_0\sigma_0,k_0) = v_{1,-1}(\en_{1,\sigma_0}(P_0))=v_{1,-1}(\st_{0,1}(P_0))<0,
$$
we conclude that the first inequality in~\eqref{pepe4} holds. In order to prove the second inequality, we define the morphism $\varphi_0\colon A_1\to A_1$ by
$$
\varphi_0(X)\coloneqq X\quad\text{and}\quad \varphi_0(Y)\coloneqq Y+\mu_0 X^{\sigma_0}.
$$
Then $(P_1,Q_1)\coloneqq (\varphi_0(P_0),\varphi_0(Q_0))$ is also a counterexample to DC. It is easy to check that $\ell_{0,1}(P_1) = \ell_{0,1}(P_0)$, and so
\begin{equation}\label{eq1}
\st_{0,1}(P_1) = \st_{0,1}(P_0) = (n_0-k_0\sigma_0,k_0).
\end{equation}
Let $\varphi_L\colon L\to L$ be the morphism defined by $\varphi_L(x) \coloneqq x$ and $\varphi_L(y)\coloneqq y+\mu_0 x^{\sigma_0}$. By~\cite{GGV1}*{Proposition~5.1}, we have
\begin{align*}
& \ell_{\rho,\sigma}(P_1) = \ell_{\rho,\sigma}(P_0)\quad\text{for $(\rho,\sigma)$ such that $(\rho_0,\sigma_0)<(\rho,\sigma)\le (0,1)$},\\
\shortintertext{and}
&\ell_{1,\sigma_0}(P_1) = \varphi_L(\ell_{1,\sigma_0}(P_0)) = \lambda_0x^{n_0-k_0\sigma_0}y^{k_0}.
\end{align*}
So, if we write $(\rho_1,\sigma_1)\coloneqq \Pred_{P_1}(0,1)$, then $(\rho_1,\sigma_1)<(\rho_0,\sigma_0)$. If $(\rho_1,\sigma_1)\le (1,1)$, then
$$
\en_{1,1}(P_1) = \st_{0,1}(P_1) = (n_0-k_0\sigma_0,k_0),
$$
and so $n_0-k_0\sigma_0+k_0 = v_{1,1}(P_1)\ge 16$, by~\cite{GGV1}*{Corollary~7.4}. On the other hand, if $(1,1)<(\rho_1,\sigma_1)<(0,1)$, then $\sigma_1>\rho_1>0$ and consequently, by~\cite{GGV1}*{Lemma~6.4}, we have $\rho_1=1$, $\st_{1,\sigma_1}(P_1) = (n_1,0)$ where $n_1\coloneqq v_{1,1}(P_1)$, and
$$
\ell_{1,\sigma_1}(P_1)\! =\! x^{n_1}f_{P_1,1,\sigma_1}(x^{-\sigma_1}y)\! =\! x^{n_1} \lambda_1(x^{-\sigma_1}y-\mu_1)^{k_1}\! =\! x^{n_1-\sigma_1 k_1} \lambda_1(y-\mu_1 x^{\sigma_1})^{k_1},
$$
where $\lambda_1,\mu_1\in K^{\times}$ and $k_1\in \mathds{N}$. We now define the morphism $\varphi_1\colon A_1\to A_1$ by
$$
\varphi_1(X)\coloneqq X\quad\text{and}\quad \varphi_1(Y)\coloneqq Y+\mu_1 X^{\sigma_1}.
$$
Then $(P_2,Q_2)\coloneqq (\varphi_1(P_1),\varphi_1(Q_1))$ is also a counterexample to DC. It is easy to check that $\ell_{0,1}(P_2) = \ell_{0,1}(P_1)$, and so
\begin{equation}\label{eq2}
\st_{0,1}(P_2) = \st_{0,1}(P_1) = (n_0-k_0\sigma_0,k_0),
\end{equation}
by~\eqref{eq1}. Write $(\rho_2,\sigma_2)\coloneqq \Pred_{P_2}(0,1)$. Arguing as above we obtain that
$$
(\rho_2,\sigma_2)\le (1,1)\quad\text{or}\quad (\rho_2,\sigma_2) = (1,\sigma_2) < (1,\sigma_1) = (\rho_1,\sigma_1).
$$
In the first case
$$
\en_{1,1}(P_2) = \st_{0,1}(P_2) = (n_0-k_0\sigma_0,k_0),
$$
and so $n_0-k_0\sigma_0+k_0 = v_{1,1}(P_1)\ge 16$, by~\cite{GGV1}*{Corollary~7.4}. In the second case we continue this construction and obtain $(1,\sigma_0)>(1,\sigma_1)>(1,\sigma_2)>\dots$, until for some $m$ we have $\Pred_{P_m}(1,1)= (\rho_m,\sigma_m)\le (1,1)$. Since
 $\sigma_0>\sigma_1>\sigma_2>\dots $ can have only finite length, this necessarily happens. But then, by~\cite{GGV1}*{Corollary~7.4}, we have $n_0-k_0\sigma_0+k_0 =\deg(P_m)\ge 16$, as we want.
$$
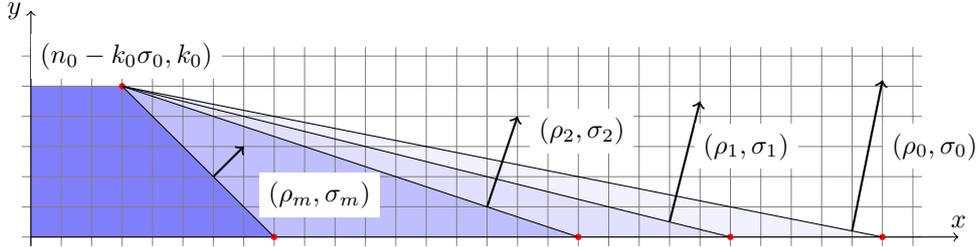
\begin{tikzpicture}[scale=0.4]
\fill[blue!50] (0,0) -- (8,0) -- (3,5)--(0,5);
\fill[blue!25] (8,0) -- (3,5) -- (18,0);
\fill[blue!12] (18,0) -- (3,5) -- (23,0);
\fill[blue!5] (23,0) -- (3,5) -- (28,0);
\draw[step=1cm,gray,very thin] (-0.3,-0.3) grid (29.3,6.3);
\draw [->] (0,0)--(30.5,0) node[anchor=south]{$x$};
\draw [->] (0,-0.3) --(0,7.5) node[anchor=east]{$y$};
\fill[red] (3,5) circle (3pt)
(8,0) circle (3pt)
(18,0) circle (3pt)
(23,0) circle (3pt)
(28,0) circle (3pt);
\draw[thick,->](6,2)--(7,3);
\draw[thick,->](27,0.2)--(28,5.2);
\draw[thick,->](21,0.5)--(22,4.5);
\draw[thick,->](15,1)--(16,4);
\draw[-](3,5)--(8,0);
\draw[-](3,5)--(18,0);
\draw[-](3,5)--(23,0);
\draw[-](3,5)--(28,0);
\draw(0,6) node[fill=white, above=0pt, right=0pt]{$(n_0-k_0\sigma_0,k_0)$};
\draw(28,3) node[fill=white, above=0pt, right=0pt]{$(\rho_0,\sigma_0)$};
\draw(21.8,3) node[fill=white, above=0pt, right=0pt]{$(\rho_1,\sigma_1)$};
\draw(16.4,3.5) node[fill=white, above=0pt, right=0pt]{$(\rho_2,\sigma_2)$};
\draw(7.5,1.4) node[fill=white, above=0pt, right=0pt]{$(\rho_m,\sigma_m)$};
\end{tikzpicture}
$$

\smallskip

\noindent\textsc{Case 7)}: By Proposition~\ref{F existe}, there exist $\mu\in K\setminus\{0\}$, $k\ge 2$ and $(1,1)$-ho\-mo\-ge\-neous elements $R,F\in A_1\setminus\{0\}$, such that
\begin{equation}\label{eq3}
\ell_{1,1}(P)=\mu \psi(R)^k = \mu\ell_{1,1}(R)^k,\quad v_{1,1}(F)=2\quad\text{and}\quad [R,F]_{1,1}=\psi(R).
\end{equation}
Moreover by hypothesis, there exists $r\in \mathds{N}$, such that $(r,0),(0,r)\in \Supp(R)$.~Then, by~\eqref{eq3} and~\cite{GGV1}*{Proposition~1.9(4)}, we have $\st_{1,1}(P) = k\st_{1,1}(R) = (kr,0)$. We claim that $t(f_R)\ge 3$. Write $F=a X^2+b XY+c Y^2$. Assume by contradiction that $t(f_R)=\#\Supp(R)=2$, and so $R=\lambda_0 X^r+\lambda_1 Y^r$, for some \hbox{$\lambda_0,\lambda_1\in K^{\times}$}. Comparing the coefficients of $x^r$ and $y^r$ in the equality $[R,F]_{1,1}=\psi(R)$, we obtain that $-br\lambda_0 = \lambda_0$ and $br\lambda_1 = \lambda_1$, which is impossible. Consequently, by Proposition~\ref{potencia de tres}, we have either
$$
 t(f_R^k)\ge 5\qquad\text{or}\qquad f_R\simeq 1+x-\frac 12 x^2.
$$
We claim that $f_R\simeq 1+x-\frac 12 x^2$ is impossible. By Remark~\ref{soporte da cota}, we have
\begin{equation}\label{eq4}
x^{kr}f_P(z)=\ell_{1,1}(P)=\mu \ell_{1,1}(R)^k=\mu x^{rk}f_R^k(z),\quad\text{where $z\coloneqq x^{-1}y$.}
\end{equation}
Hence, $f_R$ has exactly two different linear factors (the same as $f_P$), and so, by Proposition~\ref{factores lineales distintos}, necessarily $r = \deg(f_R) =2$. Since, by~\cite{GGV1}*{Corollary~4.4(1)}, the polynomials $f_F$ and $f_R$ have the same linear factors, $f_R=\lambda f_F$, for~some $\lambda\in K^{\times}$, which implies $R=\lambda F$ (because $\st(R) = (2,0) = \st(F)$). But this is im\-pos\-sible, since $[\lambda F,F]_{1,1}=0\ne \psi(R)$.~Con\-sequently $t(f_R^k)\ge 5$, and so
$$
m(P) \ge t(f_P) = t(f_R^k)\ge 5,
$$
where the first inequality holds by~\eqref{m contra t} and the equality follows from~\eqref{eq4}.

\smallskip

\noindent\textsc{Cases 1), 2) and 5)}: In these three cases we have $v_{1,-1}(\en_{1,0}(P))<0$. We claim that there is a direction $(\rho,\sigma)$, such that
\begin{equation}\label{cond (rho,sigma)}
\!(1,-1)<(\rho,\sigma)\le (1,0),\!\quad v_{1,-1}(\st_{\rho,\sigma}(P))>0\!\quad \text{and}\!\quad v_{1,-1}(\en_{\rho,\sigma}(P))<0.
\end{equation}
In order to prove the claim, we first note that, $\Dir(P)\cap ](1,-1),(1,0)]\ne \emptyset$, since otherwise $\en_{1,0}(P)=\st_{1,0}(P)=\en_{1,-1}(P)$, and so
$$
v_{1,-1}(P)=v_{1,-1}(\en_{1,-1}(P))=v_{1,-1}(\en_{1,0}(P))<0,
$$
which contradicts Proposition~\ref{no estan en D mayor o igual a cero}. Hence, we can write
$$
\Dir(P)\cap ](1,-1),(1,0)]=\{(\rho_1,\sigma_1)<(\rho_2,\sigma_2)<\dots <(\rho_k,\sigma_k)\},
$$
Note that $v_{1,-1}(\st_{\rho_j,\sigma_j}(P))>v_{1,-1}(\en_{\rho_j,\sigma_j}(P))$ for each $1\le j\le k$. Since
$$
v_{1,-1}(\st_{\rho_1,\sigma_1}(P)) = v_{1,-1}(\en_{1,-1}(P)) = v_{1,-1}(P) >0
$$
and, by hypothesis,
$$
v_{1,-1}(\en_{\rho_k,\sigma_k}(P)) = v_{1,-1}(\en_{1,0}(P)) <0,
$$
there exists $j_0$ such that
$$
v_{1,-1}(\st_{\rho_{j_0},\sigma_{j_0}}(P))>0\quad\text{and}\quad v_{1,-1}(\en_{\rho_{j_0},\sigma_{j_0}}(P))\le 0.
$$
But the condition $v_{1,-1}(\en_{\rho_{j_0},\sigma_{j_0}}(P))= 0$ leads to $\en_{\rho_{j_0},\sigma_{j_0}}(P)\sim (1,1)$, which is impossible by \cite{GGV1}*{Theorem~4.1(3)}. Setting $(\rho,\sigma)\coloneqq (\rho_{j_0},\sigma_{j_0})$, this proves~\eqref{cond (rho,sigma)}. By Proposition~\ref{F existe}, there is $\mu\in K^{\times}$, $k\ge 2$ and $(\rho,\sigma)$-homogeneous elements $R$ and $F$, such that
\begin{equation}\label{cond 1}
\!\ell_{\rho,\sigma}(P)=\mu \psi(R)^k= \mu\ell_{\rho,\sigma}(R)^k,\!\quad [R,F]_{\rho,\sigma}=\psi(R)\!\quad\text{and}\!\quad v_{\rho,\sigma}(F)=\rho+\sigma.
\end{equation}
Consequently
\begin{equation}\label{cond 2}
\st_{\rho,\sigma}(P) = k\st_{\rho,\sigma}(R)\quad\text{and}\quad \en_{\rho,\sigma}(P) = k\en_{\rho,\sigma}(R).
\end{equation}
Let $f_P=f_{P,\rho,\sigma}$, $f_R= f_{R,\rho,\sigma}$ and $f_F = f_{F,\rho,\sigma}$ be as in Remark~\ref{soporte da cota}. Similarly as in~\eqref{eq4}, we have
\begin{equation}\label{eq7}
f_P(z) = \mu f_R^k(z)\quad\text{where $z\coloneqq x^{-\sigma/\rho}y$.}
\end{equation}
We assert that
\begin{equation}\label{R cumple condiciones}
\#\Supp(\ell_{\rho,\sigma}(R))\ge 3\quad\text{and}\quad f_R\not\simeq 1+x-\frac 12 x^{2}.
\end{equation}
In order to prove the assertion, note that by~\cite{GGV1}*{Corollary~4.4} every linear factor~of~$f_R$, which is necessarily a linear factor of $f_P$, divides $f_F$. Hence, by Proposition~\ref{factores lineales distintos}, if~\eqref{R cumple condiciones} is false, then $f_R$ is separable and so $f_R\mid f_F$. Consequently, to prove~\eqref{R cumple condiciones} it will be sufficient to verify that $\deg(f_R)>\deg(f_F)$.

By~\eqref{cond (rho,sigma)} we know that $\sigma\le 0$. We divide the proof in two cases:

\smallskip

\noindent\textsc{Case $\sigma=0$:}\enspace In this case $(\rho,\sigma)=(1,0)$. Since $v_{1,0}(\st_{1,0}(F)) = v_{1,0}(F) = 1$, we have $\st_{1,0}(F)=(1,l)$ with $l\in \mathds{N}_0$. Hence, by equalities~\eqref{calculo de st y en} and~\eqref{relacion entre ell y f}, we have $\psi(F)= xy^l f_F(y)$ and $\en_{1,0}(F)=(1,l+\deg(f_F))$. By~\eqref{cond (rho,sigma)} and~\eqref{cond 2}, we know that $v_{1,-1}(\st_{10}(R))>0$ and $v_{1,-1}(\en_{10}(R))<0$, which implies that $\deg(f_R)\ge 2$. So, if $\deg(f_F)\le 1$, then $\deg(f_R)>\deg(f_F)$ and condition~\eqref{R cumple condiciones} is satisfied. Consequently, we can assume that $\en_{1,0}(F)=(1,r)$, for some $r\coloneqq l+\deg(f_F)\ge 2$. By~\eqref{cond 2} and~\cite{GGV1}*{Theorem~4.1(2)} (which applies by Proposition~\ref{v mayor que cero}), we have $\en_{1,0}(F)\sim \en_{1,0}(P)\sim \en_{1,0}(R)$. Hence $\en_{1,0}(R)=(i,ir)$, for some $i\in \mathds{N}$. By~\eqref{calculo de st y en}, we have
\begin{equation}\label{para calcular deg fR}
\st_{1,0}(R)=(i,ir)-\deg(f_R)(0,1)=(i,j).
\end{equation}
So, there exists $j\in \mathds{N}_0$, such that $(i,j)=\st_{1,0}(R)$, and so $\psi(R) = x^iy^jf_R(y)$. Note that $j<i$, since $v_{1,-1}(\st_{1,0}(R))>0$. Moreover $R = X^iY^jf_R(Y)$.

We claim that if $i=1$, then $f_R(y)$ has no linear factor $y-\lambda$ with $\lambda\ne 0$ and multiplicity one. In fact, otherwise, we write $f_R(y) = (y-\lambda)\bar{f}(y)$, with $\bar{f}(\lambda)\ne 0$, and we define $\varphi\colon A_1\to A_1$ and $\varphi_L\colon L\to L$ by
$$
\varphi(X)\coloneqq X,\quad \varphi(Y)\coloneqq Y+\lambda,\quad \varphi_L(x)\coloneqq x\quad\text{and}\quad \varphi_L(y)\coloneqq y+\lambda.
$$
By \cite{GGV1}*{Proposition~5.1}, we have
$$
\ell_{1,0}(\varphi(R)) = \varphi_L(\ell_{1,0}(R)) = x(y+\lambda)^jy\bar{f}(y+\lambda),
$$
which implies that $(1,1)\in \Supp(\ell_{1,0}(\varphi(R)))$ and $(1,0)\notin \Supp(\ell_{1,0}(\varphi(R)))$. Conse\-quently $\st_{1,0}(\varphi(R))=(1,1)$. Since, by~\cite{GGV1}*{Proposition~5.1} and~\eqref{cond 1},
$$
\ell_{1,0}(\varphi(P))= \varphi_L(\ell_{1,0}(P)) = \mu \varphi_L(\ell_{1,0}(R))^k = \mu \ell_{1,0}(\varphi(R))^k,
$$
we have $\st_{1,0}(\varphi(P))\sim (1,1)$, which is impossible by~\cite{GGV1}*{Theorem~4.1(3)} (this~the\-orem applies because $[\varphi(P),\varphi(Q)]=1$ and $v_{1,0}(\varphi(P)) = v_{1,0}(\varphi(P))>0$, by~\cite{GGV1}*{Propo\-sition~5.1} and Proposition~1.4). Hence, the claim is true, and by~Proposi\-tion~\ref{factores lineales distintos}, we obtain that~\eqref{R cumple condiciones} holds when~$i=1$.

Assume now that $i\ge 2$. Since $j<i$ and $r\ge 2$, we have
$$
\deg(f_R)=ir-j>ir-i=r-1+(i-1)(r-1)\ge r\ge \deg(f_F),
$$
where the first equality follows from~\eqref{para calcular deg fR}. As we saw above, this inequality suf\-fices to conclude that~\eqref{R cumple condiciones} holds.

\smallskip

\textsc{Case $\sigma<0$}: Since $F$ is $(\rho,\sigma)$-homogeneous, $v_{\rho,\sigma}(F) = \rho+\sigma$ and $\rho>1$ (because $\rho>-\sigma>0$), we know that
$$
\Supp(F)\subseteq (1,1)+\mathds{N}_0(-\sigma,\rho),
$$
and so $v_{1,-1}(\st_{\rho,\sigma}(F))\le 0$. Since $v_{1,-1}(\st_{\rho,\sigma}(P))>0$ (by condition~\eqref{cond (rho,sigma)}), we have $\st_{\rho,\sigma}(P)\nsim \st_{\rho,\sigma}(F)$ and so, by~\cite{GGV1}*{Theorem~4.1(1)}, we have $\st_{\rho,\sigma}(F)=(1,1)$. Since $(\rho,\sigma)\in \Dir(P)$, we have $\# \factors(f_P)\ge1$, which, by~\cite{GGV1}*{Corollary~4.4(1)}, implies that $f_F$ has at least one linear factor. Hence $F$ is not a monomial, and so
\begin{equation}\label{eq5}
\en_{\rho,\sigma}(F)=(1,1)+s(-\sigma,\rho)\quad\text{for some $s>0$.}
\end{equation}
Then, by~\eqref{cond 2} and~\cite{GGV1}*{Theorem~4.1(2)}, we have $\en_{\rho,\sigma}(F)\sim \en_{\rho,\sigma}(P)\sim \en_{\rho,\sigma}(R)$. Since $v_{\rho,\sigma}(F)>0$ and $v_{\rho,\sigma}(R) = \frac{1}{k} v_{\rho,\sigma}(P)>0$, we have
\begin{equation}\label{eq6}
\en_{\rho,\sigma}(F)=\nu \en_{\rho,\sigma}(R),\quad\text{for some $\nu>0$.}
\end{equation}
Let $(i,j)\coloneqq \st_{\rho,\sigma}(R)$. By~\eqref{cond (rho,sigma)} and~\eqref{cond 2}, we have \hbox{$i-j = \frac{1}{k} v_{1,-1}(\st_{\rho,\sigma}(R))>0$}. Since, moreover, $\sigma<0$ and $\rho+\sigma>0$,
$$
v_{\rho,\sigma}(R)=v_{\rho,\sigma}(i,j)=\rho i+\sigma j> i(\rho+\sigma)\ge \rho+\sigma=v_{\rho,\sigma}(F) = \nu v_{\rho,\sigma}(R).
$$
Since $v_{\rho,\sigma}(R)>0$, this implies that $0<\nu<1$. Write $\en_{\rho,\sigma}(R)=(i,j)+r(-\sigma,\rho)$. We claim that $r>s$. In fact, by~\eqref{eq5} and~\eqref{eq6}, we have
$$
\nu(j-i+r(\rho+\sigma))=\nu v_{-1,1}(\en_{\rho,\sigma}(R))= v_{-1,1}(\en_{\rho,\sigma}(F))= s(\rho+\sigma).
$$
Since $0<\nu<1$ and $i>j$, this implies that
$$
r(\rho+\sigma)>\nu r(\rho+\sigma)>\nu(j-i+r(\rho+\sigma))= s(\rho+\sigma),
$$
and so $r>s$, as desired. Consequently, $\deg(f_R)>\deg(f_F)$, because $\deg(f_R)=r\rho$ and $\deg(f_F)=s\rho$ (by the first equality in~\eqref{relacion entre ell y f}). As we saw above, this inequality suf\-fices to conclude that~\eqref{R cumple condiciones} holds. Thus, we have proved~\eqref{R cumple condiciones} in all cases.

By~\eqref{m contra t} and the inequality in~\eqref{R cumple condiciones}, we have $t(f_R)\ge 3$. By~\eqref{eq7} we know~that there exists $k\ge 2$ and $\mu\in \mathds{Q}^{\times}$ such that $f_P = \mu f_R^k$. Hence, by Proposition~\ref{potencia de tres} and the second condition in~\eqref{R cumple condiciones}, we have $t(f_P) = t(f_R^k)\ge 5$. Finally, by~\eqref{m contra t}~we~con\-clude that $m(P)\ge t(f_P)\ge 5$.
\end{proof}

\begin{theorem} \label{central} If $(P,Q)$ is a counterexample to the DC, then $m(P)\ge 5$.
\end{theorem}

\begin{proof} The strategy of the proof is to verify that the cases in Proposition~\ref{proposicion principal} cover all cases of Remark~\ref{casos soporte}. We first consider Case~1a) of that remark. Modifying if necessary $P$ and $Q$ via the morphism $\tau\colon A_1\to A_1$, given by $\tau(X)\coloneqq Y$ and $\tau(Y)\coloneqq -X$ (which preserves the mass $m(P)$), we can assume that $j\ge i$. Actually $j>i$, because $i=j$ is impossible by~\cite{GGV1}*{Theorem~4.1(3)}. So $v_{1,-1}(\en_{1,1}(P))<0$. We claim that $P$ is subrectangular. Since $P\notin K[X]\cup K[Y]$, by~\cite{GGV1}*{Lemma~6.5} we have to prove that
$$
\Dir(P)\cap I=\emptyset, \quad \text{where}\quad I=\{ (\rho,\sigma)\in \mathfrak{V}: (1,0)<(\rho,\sigma)<(0,1) \}.
$$
Assume by contradiction that $(\rho,\sigma)\in \Dir(P)$ and that $(1,0)<(\rho,\sigma)<(0,1)$. Since $\ell_{1,1}(P)$ is a monomial, necessarily $(1,0)<(\rho,\sigma)<(1,1)$ or $(1,1)<(\rho,\sigma)<(0,1)$. In other words $0<\sigma<\rho$ or $0<\rho<\sigma$. In the first case, by~\cite{GGV1}*{Lemma~6.4(2)} we have $\Supp(\ell_{1,1}(P)=\{(0,v_{1,1}(P))\}$ which contradicts that $\Supp(\ell_{1,1}(P)=(i,j)$ with $i,j>0$. The second case follows similarly, but using item~1) of~\cite{GGV1}*{Lemma~6.4} instead of item~2). Thus $P$ is subrectangular, and so $\en_{1,1}(P) = \en_{1,0}(P)$. But then $v_{1,-1}(\en_{1,0}(P))<0$ and Case~1a) is covered.

Case~1b) of Remark~\ref{casos soporte} follows from Cases~2) and~3) of Proposition~\ref{proposicion principal}, because $v_{1,-1}(\en_{1,0}(P))=0$ is impossible by~\cite{GGV1}*{Theorem~4.1(3)}; while Case~1c) reduces to Case~1b) using the morphism $\tau\colon A_1\to A_1$. By the same argument Case~2b) reduces to Case~2c). We next prove Case~2c). For this note that
$$
(\rho,\sigma)\coloneqq \Pred_{1,1}(P)\le (1,0),
$$
because otherwise $0<\sigma<\rho$ and~\cite{GGV1}*{Lemma~6.4(2)} impliess that $\ell_{1,1}(P)$ is a monomial, which is false. Hence $\en_{1,0}(P) = \st_{1,1}(P)$, and so, Cases~5) and~6) of Proposition~\ref{proposicion principal} cover Case~2c), because $v_{1,-1}(\en_{1,0}(P))=0$ is impossible. Finaly Cases~2a) and~3) of Remark~\ref{casos soporte} are covered by Cases~4) and~7) of Proposition~\ref{proposicion principal}, respectively.
\end{proof}

\begin{corollary}[Compare with~\cite{HT}*{Theorem~1.2}]\label{resultado principal} Assume that $P,Q\in A_1$ and $[P,Q]=1$. If $P$ is a sum of not more than~$4$ homogeneous elements of $A_1$, then $P$ and $Q$ generate $A_1$.
\end{corollary}

\begin{remark} For $(\rho,\sigma)=(3,-1)$ the $(\rho,\sigma)$-homogeneous elements
$$
R=X+2X^2Y^3+X^3Y^6\quad \text{and}\quad F=-XY-X^2Y^4
$$
satisfy $[R,F]_{\rho,\sigma}=\ell_{\rho,\sigma}(R)$ and $m(R^2)=5$. Thus, in order to improve the lower bound of~$5$ we just achieved, one has to change the strategy. For example, one can consider the list of possible smallest counterexamples as in~\cite{GGV2}*{Remark~7.9},  or the more detailed list given in~\cite{GGHV}*{Section~5}, and make a thorough analysis taking account the small differences between  the Dixmier Conjecture and the plane Jacobian conjecture in this geometric approach. One also has to consider that one of the main tools in this approach,~\cite{GGV1}*{Proposition~3.9} (which ``cuts'' the shape of the support), does not preserve the mass of the elements.
\end{remark}

\end{document}